\documentclass[reqno,11pt]{amsart}
\usepackage{graphicx,amsfonts,amssymb,amsmath,amsthm,amscd}
\usepackage[all]{xy}
\usepackage{color}
\usepackage{hyperref}
\hypersetup{colorlinks,linkcolor=magenta,citecolor=blue,filecolor=blue,urlcolor=blue}

\theoremstyle{plain} 
\newcounter{statement}[section]

\newtheorem{theorem}    [statement]{Theorem}
\newtheorem{lemma}      [statement]{Lemma}
\newtheorem{cor}  [statement]{Corollary}
\newtheorem{prop}[statement]{Proposition}

\newcounter{intro}

\newtheorem{intro_theorem}[intro]{Theorem}
\newtheorem{intro_cor}[intro]{Corollary}

\newtheorem*{theorem*}           {Theorem}
\newtheorem*{conjecture*}           {Conjecture}

\theoremstyle{definition}
\newtheorem{definition} [statement]{Definition}

\newtheorem*{definition*}           {Definition}

\theoremstyle{remark}
\newtheorem{remark}[statement]{Remark}

\definecolor{prpl}{rgb}{0.7, 0.0, 0.7}
\newenvironment{JT}{\noindent \color{prpl}{\bf JT:} \footnotesize}{} \newenvironment{BB}{\noindent \color{blue}{\bf BB:} \footnotesize}{}

\def\Sym{\operatorname{Sym}}

\def\Hom{\operatorname{Hom}}

\def\Cl{\operatorname{Cl}}

\def\End{\operatorname{End}}

\def\SL{\operatorname{SL}}

\def\mult{\operatorname{mult}}
\def\vol{\operatorname{vol}}

\def\Nm{\operatorname{Nm}}
\def\Mod{\operatorname{mod}}
\def\arcosh{\operatorname{arcosh}}

\def\Imag{\operatorname{Im}}
\def\ord{\operatorname{ord}}
\def\chone{\operatorname{c_1}}
\def\Supp{\operatorname{Supp}}
\def\rk{\operatorname{rk}}
\def\gon{\operatorname{gon}}
\def\Pic{\operatorname{Pic}}

\def\Ric{\operatorname{Ric}}

\def\Ram{\operatorname{Ram}}

\def\P{\mathbb{P}}
\def\Q{\mathbb{Q}}
\def\Z{\mathbb{Z}}
\def\C{\mathbb{C}}
\def\R{\mathbb{R}}
\def\G{\mathbb{G}}
\def\F{\mathbb{F}}
\def\H{\mathbb{H}}
\def\A{\mathbb{A}}

\renewcommand{\O}{\mathcal{O}}

\renewcommand{\phi}{\varphi}
\renewcommand{\tilde}[1]{\widetilde{#1}}
\renewcommand{\bar}[1]{\overline{#1}}

\def\into{\rightarrow}

\newcommand{\mat}[4]{\begin{pmatrix}#1&#2\\#3&#4\end{pmatrix}}

\newcommand{\p}{\mathfrak{p}}
\newcommand{\n}{\frak{n}}

 \title[Geometric torsion of abelian varieties with real mult.]{The geometric torsion conjecture for abelian varieties with real multiplication}
 \date{\today}
 \author{Benjamin Bakker}
 \address{B. Bakker:
 Institut f\"ur Mathematik, Humboldt-Universit\"at zu Berlin.
 }
 \email{benjamin.bakker@math.hu-berlin.de}
\author{Jacob Tsimerman}
\address{J. Tsimerman:
Mathematics Department, University of Toronto.}
\email{jacobt@math.toronto.edu}

\begin{document}

\begin{abstract}
The geometric torsion conjecture asserts that the torsion part of the Mordell--Weil group of a family of abelian varieties over a complex quasiprojective curve is uniformly bounded in terms of the genus of the curve.  We prove the conjecture for abelian varieties with real multiplication, uniformly in the field of multiplication.  Fixing the field, we furthermore show that the torsion is bounded in terms of the \emph{gonality} of the base curve, which is the closer analog of the arithmetic conjecture.  The proof is a hybrid technique employing both the hyperbolic and algebraic geometry of the toroidal compactifications of the Hilbert modular varieties $\bar X(1)$ parametrizing such abelian varieties.  We show that only finitely many torsion covers $\bar X_1(\n)$ contain $d$-gonal curves outside of the boundary for any fixed $d$.  We further show the same is true for entire curves $\C\into \bar X_1(\n)$.

\end{abstract}

\maketitle
\section*{Statement of Results}
For any elliptic curve $E/\Q$, the group of rational points $E(\Q)$ is finitely generated by Mordell's theorem.  The free part behaves wildly; it is expected that there are elliptic curves $E/\Q$ with arbitrarily large rank $\rk E(\Q)$, and the record to date is an elliptic curve $E$ with $\rk E(\Q)\geq 28$ found by Elkies.  On the other hand, by a celebrated theorem of Mazur \cite{mazur} the torsion part $E(\Q)_{tor}$ is uniformly bounded:
\begin{theorem*}[Mazur] For any elliptic curve $E/\Q$, $|E(\Q)_{tor}|\leq 16$.
\end{theorem*}
Mazur's theorem was subsequently generalized to arbitrary number fields $K/\Q$ by Merel \cite{Merel} (building on partial results of \cite{kam}) who showed a stronger uniformity:  there is an integer $N=N(d)$ such that for any degree $d$ number field $K$ and any elliptic curve $E/K$, every $K$-rational torsion point has order dividing $N$, \emph{i.e.} $E(K)_{tor}\subset E(K)[N]$.

Similarly, it is expected that the torsion part of the Mordell--Weil group of an abelian variety $A/K$ is uniformly bounded, though there are few results in this direction.  The same question can be asked for $K=k(C)$ the function field of a curve $C$ over any field $k$, and though $k=\bar\F_p$ is most closely analogous to the number field case, $k=\bar\Q$ is also interesting:

\begin{conjecture*}[Geometric torsion conjecture] \label{gtc} Let $k$ be an algebraically closed field of characteristic 0.  There is an integer $N=N(g,n)$ such that for any quasi-projective genus $g$ curve $C/k$ and any family of $n$-dimensional abelian varieties $A/C$ with no isotrivial part, the torsion of the Mordell--Weil group is uniformly bounded:
\[A(C)_{tor}\subset A(C)[N]\]
\end{conjecture*}
Here $A(C)$ is the group of rational sections.  By the standard argument, it suffices to consider $k=\C$, and we do so for the remainder.  The geometric conjecture is also largely open, though some recent progress has been made by Cadoret and Tamagawa \cite{cadtampp,cadtamppk}.  Their technique however only applies over a \emph{fixed} base $C$.  

The main goal of this paper is to prove the conjecture for abelian varieties with real multiplication:
\begin{intro_theorem}\label{rmmain}There is an integer $N=N(g,n)$ such that for any quasi-projective genus $g$ complex curve $C$ and any nonisotrival family of $n$-dimensional abelian varieties $A/C$ with real multiplication, the torsion part of the Mordell--Weil group is uniformly bounded:
\[A(C)_{tor}\subset A(C)[N]\]

\end{intro_theorem}
An abelian variety with real multiplication is an $n$-dimensional abelian variety $A$ together with an injection $\O_F\into \End(A)$ of the ring of integers $\O_F$ in a totally real field $F/\Q$ of degree $n$.  Note that the constant in Theorem \ref{rmmain} only depends on $F$ through the degree $n$, and thus has the same uniformity as in the conjecture.  Note also that there are no reduction hypotheses in Theorem \ref{rmmain}, so the result is really a statement about abelian varieties over the function fields of complex curves:

\begin{intro_cor}\label{field}For any genus $g$ complex curve $C/k$ and any nonisotrivial abelian variety $A$ over $K=k(C)$ with real multiplication,
\[A(K)_{\textrm{tor}}\subset A(C)[N]\]
\end{intro_cor}

As a consequence, the torsion part of the Mordell--Weil group can only be finitely many groups, depending only on the genus of $C$.  We also prove a version of the conjecture uniformly in the \emph{gonality} of the base curve, at the expense of a dependence on the field of multiplication:

\begin{intro_theorem}\label{rmmaingon}Fix a degree $n$ totally real field $F$.  There is an integer $N=N(d,F)$ such that for any quasi-projective $d$-gonal complex curve $C$ and any nonisotrival family of abelian varieties $A/C$ with real multiplication by $\O_F$, the torsion part of the Mordell--Weil group is uniformly bounded:
\[A(C)_{tor}\subset A(C)[N]\]

\end{intro_theorem}
Of course, there is a corresponding version of Corollary \ref{field}.  Uniformity in the \emph{gonality} of the base curve $C$ is strictly stronger than uniformity in the genus; it is the correct function field analog of the degree of the number field in Merel's theorem.  Recall that, for a curve $C$ the gonality $\gon(C)$ is the minimum degree of a (dominant) map to $\P^1$, and since the map $C^{(g+1)}\into \Pic^{g+1}(C)$ has relative dimension 1, we at least\footnote{In fact, $\gon(C)\leq \lfloor\frac{g(C)+1}{2}\rfloor$, with equality for $C$ generic of genus $g(C)$.} have $\gon(C)\leq g(C)+1$.  On the other hand, for any $d>0$ there are $d$-gonal curves of genus $g$ for all $g\geq 2d-3$.  Similarly, while there are infinitely many $d$-gonal curves $C/\F_q$, there are only finitely many for any fixed genus.  This should be viewed as the function field analog of the fact that there are infinitely many degree $d$ number fields $K/\Q$ but only finitely many of bounded discriminant.

Our proofs of Theorems \ref{rmmain} and \ref{rmmaingon} are geometric.  For a fixed $F$, there is a Hilbert modular variety $X(1)$ parametrizing abelian varieties with real multiplication by $\O_F$, and a map $C\into X(1)$ is equivalent to a family of such abelian varieties over $C$.  For any ideal $\frak{n}\subset \O_F$, there is a level cover $X_1(\frak{n})$ parametrizing abelian varieties $A$ with real multiplication by $\O_F$ together with a point $x\in A$ whose annihilator is $\frak{n}$.  To prove Theorem \ref{rmmain}, we show that $X_1(\frak{n})$ uniformly contains no genus $g$ curves for $|\Nm\frak{n}|$ large.

\begin{intro_theorem}  \label{nogon}For each $\frak\n \subset\O_F$ let $X_1(\frak\n)^*$ be the Baily--Borel compactification of the $\frak n$-torsion level cover of the Hilbert modular variety $X(1)$.  Then for any $g$, $X_1(\n)^*$ contains no genus $g$ curves for all but finitely many $\n$, uniformly for all $X(1)$ of a fixed dimension.  For a fixed $X(1)$, the same is true of $d$-gonal curves.
\end{intro_theorem}

In so doing, we prove some results about the geometry of toroidal compactifications $\bar X$ of varieties uniformized by $\H^n$ that are interesting in and of themselves.  The core idea is to prove a bound relating the volume of a curve in $\bar X$ to its multiplicity along the boundary (see Proposition \ref{hwangtocusp}).  For larger arithmetic lattices, intersection with the boundary comes at the price of more volume, which in turn implies larger genus.  This in particular implies that the slice of the ``ample modulo the boundary" cone generated by the canonical bundle $K_{\bar X} $ and the components of the boundary grows with the ``size" of the lattice.  For instance, we have the
\begin{intro_theorem} \label{getsample} For $\bar X_1(\n)$ a smooth toroidal compactification of $X_1(\n)$, let $D$ be the boundary divisor.  Then for any $\lambda>0$, $K_{\bar X_1(\frak\n)}-(\lambda-1) D$ is ample modulo the boundary provided $|\Nm(\n)|>(\frac{2\pi\lambda}{n})^{2n}$.
\end{intro_theorem}

\noindent Here we assume $X_1(\n)$ has no elliptic points, but this is uniformly true for large $|\Nm(n)|$ (see Lemma \ref{noell}).  Recall that a divisor $L$ is ample modulo a simple normal crossing divisor $D$ if some multiple of $L$ induces a rational map to $\P^N$ which is an embedding on the complement of $D$.  

For any toroidal compactification $\bar X$ of the quotient $X=\Gamma\backslash\Omega$ of a bounded symmetric domain $\Omega$ by an arithmetic lattice $\Gamma$, there is \emph{some} \'etale cover $X'$ of $X$ for which $K_{\bar X'}-\lambda D'$ is ample modulo the boundary (\emph{cf.} \cite{amrt}).  It is not difficult to show that the same is true for \emph{full-level} covers $\bar X(\n)$ for $\n$ sufficiently large (\emph{e.g.} \cite{hwangtolevel}) because $\bar X(\n)\into \bar X(1)$ ramifies to high order along every boundary component.  The strength of Theorem \ref{getsample} is that it applies to the \emph{torsion} level covers $\bar X_1(\n)$ which don't ramify enough for the same argument to apply, and that the statement is uniform in \emph{all} $n$-dimensional Hilbert modular varieties.

For some geometric consequences of Theorem \ref{getsample}, see Section \ref{geometry}.  In particular, it follows that $\bar X_1(\n)$ is of general type for large $|\Nm(\n)|$, and the Green--Griffiths conjecture then predicts that all entire curves $\C\into \bar X_1(\n)$ have image contained in a strict algebraic subvariety (called the exceptional locus) for those $\n$.  By a theorem of Nadel \cite{nadel}, Theorem \ref{getsample} indeed implies this is the case:

\begin{intro_cor} \label{green} Every nontrivial entire curve $\C\into \bar X_1(\n)$ has image contained in the boundary provided $|\Nm(\n)|>(2\pi)^{2n}$.
\end{intro_cor}

This provides an explicit bound  in the genus 0 or 1 case of Theorem \ref{rmmain}.  The Green--Griffiths conjecture for the base Hilbert modular varieties $\bar X(1)$ has been addressed recently by \cite{rousseau} using modular forms and foliation theory, where the conjecture is proven for all but finitely many choices of $F$.  This clearly implies the conjecture for the toroidal compactification $\bar X'$ of \emph{any} cover $X'\into X(1)$, although the exceptional locus is not explicit.  Corollary \ref{green} says that the exceptional locus is in fact (contained in) the boundary sufficiently high in the torsion tower, uniformly in $X(1)$.

%

$X_1(\n)$ is defined over $\Z$, and the \emph{arithmetic} torsion conjecture for abelian varieties with real multiplication can likewise be phrased as the nonexistence of $K$-rational points of $X_1(\n)$ for all but finitely many $\n$, for any number field $K$.  Theorem \ref{nogon} and Corollary \ref{green} are geometric and analytic analogs asserting the nonexistence of rational points valued in the function field of a curve and the field of meromorphic functions on $\C$, respectively.  All three are conjecturally related:  Corollary \ref{green} implies the genus $0$ and $1$ cases of Theorem \ref{nogon}, and assuming the Bombieri--Lang conjecture, Corollary \ref{green} implies $X_1(\n)(K)$ is \emph{finite} for any number field $K$.

We finally note that Theorem \ref{rmmain} (and Corollary \ref{field}) can be made effective.  We also expect Theorem \ref{rmmaingon} to be true uniformly in $X(1)$, and that the same idea for the proof should work with some modifications.  The methods investigated here apply more generally to rank one lattices (see \cite{hyperbolic} for an application to complex ball quotients), and we expect the torsion conjecture in the case of abelian varieties parametrized by a rank one Shimura variety to be proven similarly.

\subsection*{Outline}  The proof follows the general strategy developed in \cite{BT2} to prove the geometric analog of another arithmetic uniformity conjecture, the Frey--Mazur conjecture.  In Section \ref{background} we detail the local structure of cusps of cofinite-volume quotients of $\H^n$ and their toroidal compactifications.  We then prove a volume bound on the boundary multiplicity of curves in the toroidal compactification in Section \ref{curves}.  These bounds are similar to those proven by \cite{hwangto1} for \emph{interior} points.  In Section \ref{hilbert} we show that the torsion covers of Hilbert modular varieties hyperbolically ``expand," from which it follows that the multiplicity bound of Section \ref{background} improves in the torsion tower.  We  then deduce some geometric results in Section \ref{geometry} including Theorem \ref{getsample} and Corollary \ref{green}.  In Section \ref{genussection} we assemble the previous results to prove the first part of Theorem \ref{nogon} and conclude Theorem \ref{rmmain} (and Corollary \ref{field}).  Finally, in Section \ref{gonalitysection} we prove a volume bound on the diagonal multiplicity of curves in products of Hilbert modular varieties and use it to prove the second part of Theorem \ref{nogon}, and thus Theorem \ref{rmmaingon}.  The presence of degenerations complicates the analysis of the diagonal, and the eventual bound we obtain in Theorem \ref{rmmaingon} is ineffective.

\subsection*{Acknowledgements}  The first named author would like to thank G. van der Geer and  A. J. de Jong for enlightening discussions.

\section{Lattices in $\SL_2(\R)^n$}\label{background}
Let $\H=\{z\in\C\mid \Imag z>0\}$ be the upper half-plane.  $G=\SL_2(\R)^n$ is the group of holomorphic automorphisms of $\H^n$ acting by M\"obius transformations component-wise.  A discrete cofinite-volume subgroup $\Gamma\subset G$ is \emph{nondegenerate} if it is not commensurable to a product, and in this section we will only consider such $\Gamma$.  In this case, $\Gamma$ has only isolated parabolic fixed points, called cusps, and there are only finitely many up to action by $\Gamma$.  The Baily--Borel compactification $X^*$ of the quotient $X=\Gamma\backslash\H^n$ is obtained by adding a point to compactify each equivalence class of cusps; it is a normal projective variety \cite{bb}.  Note that for $n>1$, $\Gamma$ is commensurable to an arithmetic lattice (up to conjugation) by a theorem of Marguilis \cite{marguilis}.

\subsection*{Local models of the cusp.}
$\H^n$ has a preferred cusp at $\infty:=(\infty,\ldots,\infty)\in\H^n$, whose parabolic stabilizer $U_\infty$ is the upper triangular matrices.  For $z\in\H^n$, define $N(z):=\prod_i (\Imag z_i)$.  Natural neighborhoods of the cusp $\infty$ are given by the sets
\[U(s)=\{z\in\H^n|N( z)>1/s\}\]
and we refer to $U(s)$ as the horoball around $\infty$ of depth $s$.

Given a cusp $*$ of $\Gamma$, we may move it to $\infty$ by conjugating $\Gamma$.  This is only unique up to conjugation by an upper-triangular element
\[\gamma=\mat{a_1}{b_1}{0}{a_1^{-1}}\times\cdots\times \mat{a_n}{b_n}{0}{a_n^{-1}}\]
Note that this sends the horoball $U(s)$ to $\gamma\cdot U(s)=U(s')$ with $s'=s/\prod_i a_i^2$, as $N(\gamma\cdot z)=N(z)\prod_i a_i^2$.

The stabilizer $\Gamma_\infty=\Gamma\cap U_\infty$ of $\infty$ has unipotent radical $\Lambda_\infty$ a lattice of real translations.  $N$ yields a norm form $\Nm$ on $\Lambda_\infty$, and we can identify $\H^n$ with
\[\H_\infty:=\{\zeta\in \Lambda_\infty\otimes \C\mid(\Imag \zeta)_i>0\mbox{ for all $i$}\}\subset \Lambda_\infty\otimes\C\]
with $\Lambda_\infty$ acting by translations by $\Lambda_\infty\otimes 1$.   Here $(\Imag \zeta)_i$ is the $i$th coordinate of $\Imag \zeta$, which is not defined up to conjugation by $\gamma$, as it scales by $a_i^2$, but its positivity is.  The norm on $\Lambda_\infty$ clearly scales by $\prod_ia_i^2$.

Thus, we can scale so that the shortest vector of $\Lambda_\infty$ has length $1$, and we thereby associate to any cusp $*$ of $\Gamma$ with stabilizer $\Gamma_*$ and unipotent radical $\Lambda_*$ a canonically determined norm form $\Nm_*$ on $\Lambda_*$ normalized by the condition that the length of the shortest vector is 1.  The coordinates $\sigma^i_*:\Lambda_*\into \R$ for which
\[\Nm_*(\lambda)=\prod_i\sigma^i_*(\lambda)\]
are not well-defined as they scale individually, but as above their positivity is.  We also have a canonically determined function $N_*(\zeta)=\prod_i\sigma^i_*(\Imag\zeta)$ defined on the canonical horoball $U_*(s)$, which has the form
\[ U_*(s):=\{\zeta\in\Lambda_*\otimes\C\mid N_*(\zeta)>1/s\}\]
\begin{definition}
For a cusp $*$ of $\Gamma$, let $W_*(s):=\Gamma_* \backslash U_*(s)$.  We say the horoball $U_*(s)$ is \emph{precisely invariant} if $W_*(s)$ injects into $\Gamma\backslash\H^n$.   The \emph{canonical depth} $s_*$ of $*$ is the largest $s$ such that $U_*(s)$ is precisely invariant.  When there's no chance of confusion we just refer to $s_*$ as the depth.
\end{definition}
The following lemma gives us a bound on the canonical depth.
\begin{lemma}\label{depth}
For $\Gamma\subset G$ nondegenerate, suppose $s$ is the minimal nonzero value of $\prod_i |c_i|$ over all group elements
\[\gamma=\mat{a_1}{b_1}{c_1}{d_1}\times\cdots\times \mat{a_n}{b_n}{c_n}{d_n}\in\Gamma\]
Then $U(s)$ is precisely invariant under $\Gamma$.
\end{lemma}

\begin{proof}We need to show that if $\gamma\cdot z=w$ for some $\gamma\in \Gamma$ and $z,w\in U(s)$ then $\gamma$ is upper triangular.  We have
$$N(w) = \prod_i \frac{y_i}{|c_iz_i+d_i|^2}.$$
Note that since $\Gamma$ is nondegenerate, one of the $c_i$ is zero only if all of them are zero.  If they are not all zero, then as $|c_iz_i+d_i|\geq |c_i||z_i|\geq |y_i|$ we get $$N( w)\leq \frac{1}{\prod_i |c_i|^2N(z)}\leq \frac{1}{\prod_i|c_i|}$$ which is a contradiction.

\end{proof}
%
%
%
%
%
%
\subsection*{Toroidal compactifications}  We briefly describe the geometry of the toroidal compactifications of $X$ constructed by \cite{amrt}.

For a cusp $*$ of $\Gamma$, the partial quotient of $\H^n$ by $\Lambda_*$ naturally sits in the torus $T_*=\Lambda_*\backslash\Lambda_*\otimes\C$, and there is a $\log$ map valued in $\Lambda_*\otimes \R$ defined by taking the imaginary part:
\[\xymatrix{
\H_*\ar@{^{(}->}[r] &\Lambda_*\otimes\C\ar[r]&T_*\ar[r]^{\log} &\Lambda_*\otimes\R\\
&\lambda\otimes z\ar@{|->}[rr]&&\lambda\otimes\Imag z
}\]

The coordinate ring of $T_*$ is canonically $\C[\Lambda_*^\vee]$, via the identification between $\Lambda_*^{\vee}$ and the character group $\Hom(T_*,\G_m)$.  For an element $\chi\in \Lambda_*^{\vee}$ we denote the corresponding character by the symbol $q^\chi$, concretely given by
\[q^\chi(\lambda \otimes z)=e(\chi(\lambda)z)\]
and for any $t\in T_*$, we have
\[\chi(\log(t)) =-\frac{1}{2\pi}\log|q^{\chi}(t)|\]

The function $N_*$ descends (and extends) to the torus $T_*$, and in fact further descends to $\Lambda_*\otimes\R$; it is given by
\[N_*(t)=\left(\frac{-1}{2\pi}\right)^n\prod_i \log|q^{\sigma^i_*}(t)|=\Nm_*(\log t)\]
where on the right hand side we mean $\Nm_*$ extended to $\Lambda_*\otimes \R$ in the obvious way.
The horoballs $U_*(s)$ are likewise stable under the action of $\Lambda_*$ and we let $V_*(s)$ be the image in $T_*$.  Note that $\log(V_*(s))$ lies inside the positive cone $C(\Lambda_*\otimes\R)$ defined by
\[C(\Lambda_*\otimes\R)=\{x\in \Lambda_*\otimes\R \mid \sigma^i_*(x)>0\}\]
where we extend $\sigma_*^i$ $\R$-linearly.  Importantly, $C(\Lambda_*\otimes\R)$ is \emph{not} an integral cone.

The group $\Delta_*:=\Gamma_*/\Lambda_*$ acts on $C(\Lambda_*\otimes\R)$.  A toroidal compactification at $*$ is specified by a subdivision of $C(\Lambda_*\otimes\R)$ into a fan of integral polyhedral cones $\Sigma_*=\{\tau\}$ that is stable under the action of $\Delta_*$.  The compactification is smooth if and only if each full-dimensional $\tau$ is generated by an integral basis of $\Lambda_*$, and any fan can by sufficiently subdivided to yield a smooth compactification.  For each full-dimensional cone $\tau$, this provides us with a coordinate chart of the compactification which looks like a neighborhood of $0\in\A^n$ with coordinates $q^{\chi_1},\ldots,q^{\chi_n}$, where the $\chi_i$ are the basis dual to the basis of $\Lambda_*$ given by the primitive generators of the rays of $\tau$, and in this chart the boundary is the union of the coordinate axes $q^{\chi_i}=0$.

Taking $s$ smaller than the depth of $*$, the quotient $W_*(s)=\Delta_*\backslash V_*(s)$ injects into the toroidal compactification, and $N_*$ descends to $W_*(s)$.  We therefore obtain a function $N_*:W_*(s)\into [-\infty,\infty)$ by declaring $N_*$ to be $-\infty$ along the boundary.  To compute the Lelong number of $N_*^{1/n}$, we use the $\A^n$ neighborhoods of the boundary in the toroidal compactification.  Let $\mathbf{a}=(a_1,\ldots,a_n)\in \R^n$, and for $z\in\C^n$ denote $|z|^\mathbf{a}=\prod_i|z_i|^{a_i}$.  Then we have
 \begin{lemma}\label{prelelong}
 \[\liminf_{z\into x}\frac{\prod_i \log |z|^{\mathbf{a}^{(i)}}}{\log^n|z-x|}=\prod_j\sum_{\substack{i\\x_i=0}}a_i^{(j)}\]
 \end{lemma}
 So in the surface case this means for $x_2\neq 0$
 \[\liminf_{(z_1,z_2)\into (0,x_2)}\frac{\sqrt{(a_1\log|z_1|+a_2\log |z_2|)(b_1\log|z_1|+b_2\log |z_2|)}}{\log\sqrt{|z_1|^2+|z_2-x_2|^2}}=\sqrt{a_1b_1}\]
 and
 \[\liminf_{(z_1,z_2)\into (0,0)}\frac{\sqrt{(a_1\log|z_1|+a_2\log |z_2|)(b_1\log|z_1|+b_2\log |z_2|)}}{\log\sqrt{|z_1|^2+|z_2|^2}}=\sqrt{(a_1+a_2)(b_1+b_2)}\]
For any cone $\tau\in \Sigma_*$, let $Z(\tau)$ be the corresponding $T_*$-orbit (\emph{not} the orbit closure).
 \begin{lemma}\label{lelong}The Lelong number $\nu(-N_*^{1/n},x)$ at a point $x$ in the boundary component compactifying $*$ is constant along $Z(\tau)$ and we have
 \[\nu(-N_*^{1/n},x)=\frac{1}{2\pi}\Nm_*(\lambda(\tau))^{1/n},\textrm{ for all }x\in Z(\tau)\]
 where $\lambda_i\in\Lambda_*$ are the primitive generators of the rays of $\tau$ and $\lambda(\tau)=\sum \lambda_i$.
 \end{lemma}

 In the future we denote $\nu(-N_*^{1/n},\tau):=\nu(-N_*^{1/n},x)$ for $x\in Z(\tau)$.
 \begin{proof}Let $\tau_0$ be a full-dimensional cone with associated basis $\lambda_1,\ldots,\lambda_n\in \Lambda_*$, so if $\chi_1,\ldots,\chi_n$ is the dual basis, then the corresponding coordinates are $x_i=q^{\chi_i}$.  Faces $\tau$ of $\tau_0$ are indexed by subsets $S\subset \{1,\ldots,n\}$, where the corresponding face is generated by the rays $\{\lambda_s|s\in S\}$, and the associated orbit is locally the coordinate plane given by the intersection of $x_{s}=0$ for all $s\in S$.  We have
 \[\sigma^i_*=\sum_j\sigma^i_*(\lambda_j)\chi_j\]
 and therefore in these coordinates
 \[\log|q^{\sigma^i_*}(t)|=\log\prod_j|x_j|^{\sigma^i_*(\lambda_j)}\]
 Thus, since
 \[\prod_i\sum_{s\in S}\sigma^i_*(\lambda_s)=\prod_i\sigma^i_*(\lambda)=\Nm_*(\lambda(\tau))\]
 for $\lambda(\tau)=\sum_{s\in S}\lambda_s$, the result follows from the previous lemma.

 \end{proof}
 \section{Curves in quotients of $\H^n$}\label{curves}
 \subsection*{Metric geometry}
Endow $\H$ with its hyperbolic hermitian metric $ds^2_\H$ of constant sectional curvature $-1$; explicitly, the associated K\"ahler form is \[\omega_\H=\frac{1}{2}\Imag ds_H^2=\frac{idz\wedge d\bar z}{2y^2}=i\partial\bar\partial (-2\log y)\]
Likewise endow $\H^n$ with the invariant metric
\[ds^2_{\H^n}=\sum_i \pi_i^* ds^2_{\H}\]
where $\pi_i:\H^n\into\H$ is the $i$th projection, and we again denote by $\omega_{\H^n}$ the associated K\"ahler form.  Note that $-2\log N$ (and therefore also $-2\log N_*$) is a global potential for $\omega_{\H^n}$.  The distance function we use on $\H^n$ is the Kobayashi distance---namely, the distance between $z,w\in\H^n$ is the maximum of the distances $d_\H(z_i,w_i)$ of the coordinate projections, where $d_\H$ is the usual hyperbolic distance on $\H$:
\[\left|\frac{z-w}{z-\bar w}\right|^2 = \tanh^2(d_\H(z,w)/2)\]


For $\Gamma\subset G$ a discrete nondegenerate cofinite-volume subgroup and $X=\Gamma\backslash\H^n$, let $ds^2_X$ be the induced metric with K\"ahler form $\omega_X$.  For $\bar X$ a smooth toroidal compactification of $X$, a theorem of Mumford \cite{mumfordhp} tell us that the current $[\omega_X]\in H^{1,1}(\bar X,\R)$ defined by integration against $\omega_X$ on the open part $X$ is represented by a multiple of the log-canonical bundle:
\begin{equation}\chone(K_{\bar X}+D)=\frac{1}{2\pi}[\omega_X]\in H^{1,1}(\bar X,\R)\label{GB}\end{equation}

Take $f:C\into \bar X$ a map from a smooth proper curve of genus $g(C)$ whose image is not contained in the boundary, and let $U\subset C$ be the open subset mapping to the interior $X$.

\begin{lemma}\label{schwarz}In the above situation,
\[\frac{1}{n}(K_{\bar X}-(n-1)D).C\leq2g(C)-2\]
\end{lemma}
\begin{proof}  $U$ is necessarily uniformized by $\H$, so let $ds^2_{U}$ be it's uniformized metric of constant curvature $-1$.  By Schwarz' Lemma,
\[\frac{1}{n}f^* ds^2_X  \leq ds^2_{U}\]
Integrating the K\"ahler forms, we get by Gauss--Bonnet
\begin{equation} \frac{1}{2\pi n}\vol_X(U)\leq -\chi(U)=-\chi(C)+|C\smallsetminus U|\label{inter}\end{equation}
The left-hand side is $\frac{1}{n}(K_{\bar X}+D).C$, while $|C\smallsetminus U|\leq D.C$, and the result follows.
\end{proof}
\begin{remark}  \label{orbifold}Lemma \ref{schwarz} still holds in the form of \eqref{inter} with the same proof if $\Gamma$ has elliptic elements, as long as we treat $X$ and $C$ as orbifolds and $\chi(C)$ as the orbifold euler characteristic.
\end{remark}
%
%
%
%
 \subsection*{Multiplicity bounds}  For the rest of this section, we additionally assume $\Gamma$ is torsion-free.  We begin by recalling a theorem of Hwang--To which says that the volume of a curve in a neighborhood of a point in the interior of a quotient of $\H^n$ (or in fact any bounded symmetric domain) is bounded by its multiplicity at the point.
 \begin{theorem}[See \cite{hwangto1}] \label{hwangtopoint} Let $x\in X$ and take $B(x,r)$ the Kobayashi ball around $x$ of radius\footnote{$\rho_x$ is the injectivity radius, see Section \ref{hilbert}.} $r<\rho_x$ .  Then for any irreducible $k$-dimenisonal subvariety $Y\subset X$, we have
\[\vol_X(Y\cap B(x,r))\geq \frac{(4\pi)^k}{k!}\sinh^{2k}(r/2)\cdot\mult_x Y\]
\end{theorem}
By \eqref{GB} this theorem allows us to estimate the degree of a curve $C\subset \bar X$ not meeting the boundary in terms of its multiplicity at a point $x\in X$:
\begin{equation}(K_{\bar X}+D).C=K_{\bar X}.C=\frac{1}{2\pi n}\vol_X(C\cap X)\geq \frac{2}{n}\sinh^2(\rho_x)\cdot \mult_x C\label{htgb}\end{equation}
The main goal of this section is to provide a similar volume bound on the multiplicity along the boundary, as well as a relative version.

For any cusp $*$ of $\Gamma$, denote by $\tilde W_*(s)$ the interior closure of $W_*(s)$ in $\bar X$ for $s$ smaller than the canonical depth $s_*$ of $*$.

\begin{prop}\label{relative}Let $*$ be a cusp of $\Gamma$.  Then for any irreducible $k$-dimensional analytic variety $Y$ in $\tilde W_*(s_*)$ not contained in the boundary, $$s^{-k/n}\vol_X(Y\cap  W_*(s))$$ is an increasing function of $s$ for $s<s_*$.

\end{prop}Before the proof we recall a lemma of Demailly \cite{demaillybook}.  Let $X$ be a complex manifold and $\phi:X\into[-\infty,\infty)$ a continuous plurisubharmonic function.  Define
\[B_\phi(r)=\{x\in X\mid \phi(x)<r\}\]
We say $\phi$ is semi-exhaustive if the balls $B_\phi(r)$ have compact closure in $X$.  Further, for $T$ a closed positive current of type $(k,k)$, we say $\phi$ is semi-exhaustive on $\Supp T$ if $\phi\not\equiv -\infty$ on $\Supp T$ and each $B_\phi(r)\cap \Supp T$ has compact closure.  In this case, the integral
\[\int_{B_\phi(r)}T\wedge (i\partial\bar \partial\phi)^k:=\int_{B_\phi(r)}T\wedge (i\partial\bar \partial\max(\phi,s))^k\]
is well-defined and independent of $s<r$ \cite[\S III.5]{demaillybook}.  We then have the\begin{lemma}[Formula III.5.5 of \cite{demaillybook}]\label{demaillybook}For any convex increasing function $f:\R\into\R$,
\[\int_{B(r)}T\wedge(i\partial\bar \partial f\circ\phi)^k=f'(r-0)^k\int_{B(r)}T\wedge(i\partial\bar \partial\phi)^k\]
where $f'(r-0)$ is the derivative of $f$ from the left at $r$.
\end{lemma}
We also need the
\begin{lemma}\label{psh} $-N^{1/n}$ is plurisubharmonic on $\H^n$.
\end{lemma}
\begin{proof}Let $I_n$ be the identity $n\times n$ matrix and $Z_n$ be the matrix which has 0's along the diagonal, and 1's elsewhere. Now note that $I_n+tZ_n$ is positive semi-definite exactly when
$-\frac{1}{n-1}\leq t\leq 1$---this can be seen by computing its determinant.

For any function $g:\R\into \R$, we compute by the chain rule
\begin{align*}i\partial\overline{\partial} g(N)&=g''(N)\cdot i\partial N\wedge \bar\partial N+g'(N)\cdot i\partial\bar\partial N\\
&=\frac{1}{4}N^2g''(N)(I_n+Z_n) + \frac{1}{4}Ng'(N)Z_n
\end{align*}
in the renormalized basis $(\Imag z_i) \frac{\partial}{\partial z}$.  Setting $g(x)=-x^{t}$, we get
\begin{align*}i\partial\bar\partial (-N^{t})&=-\frac{t(t-1)N^{t}}{4}\left((I_n+Z_n)+\frac{1}{t-1}Z_n\right)=\frac{t(1-t)N^{t}}{4}\left(I_n-\frac{t}{1-t}Z_n\right)
\end{align*}
which is positive semi-definite for $t= 1/n$.
\end{proof}

\begin{proof}[Proof of Proposition \ref{relative}]

\begin{align*}
\vol_X(Y\cap  W_*(s))&=\frac{1}{k!}\int_{Y\cap  W_*(s)}\omega_X^k\\
&=\frac{1}{k!}\int_{Y\cap  W_*(s)}(i\partial\bar\partial (-2\log N_*))^k\\
&=\frac{1}{k!}\int_{ W_*(s)}[Y]\wedge(i\partial\bar\partial (-2\log N_*))^k\\
&=\frac{(2n)^ks^{k/n}}{k!}\int_{ W_*(s)}[Y]\wedge(i\partial\bar\partial (-N_*^{1/n}))^k\\
&=\frac{(2n)^ks^{k/n}}{k!}\int_{Y\cap  W_*(s)}(i\partial\bar\partial (-N_*^{1/n}))^k\\
\end{align*}
As $-N_*^{1/n}$ is plurisubharmonic,
\[s^{-k/n}\vol_X(Y\cap  W_*(s))=\frac{(2n)^k}{k!}\int_{Y\cap  W_*(s)}(i\partial\bar\partial (-N_*^{1/n}))^k\]
is an increasing function of $s$.
\end{proof}
\begin{remark}\label{optimal}  Proposition \ref{relative} is also valid in the orbifold context just by pulling up to a finite neat cover.  It is optimal in the sense that for $Y$ a the projection of a linear geodesic $\H\subset\H^n$ (no coordinate of which is constant), $s^{-1/n}\vol(Y\cap  W_*(s))$ will be constant.  Indeed, in this case $N_*^{1/n}$ restricts to a multiple of $y$, which is a potential for (a multiple of) the current of integration at the cusp.
\end{remark}
%

Taking the limit of Proposition \ref{relative} as $u\into 0$ and using the results from Section \ref{background}, we obtain a multiplicity bound in the style of Theorem \ref{hwangtopoint}.
 \begin{prop}\label{hwangtocusp}  Let $*$ be a cusp of $\Gamma$ with canonical depth $s_*$.  Then for any irreducible analytic curve $C$ in $\tilde W_*(s_*)$ not contained in the boundary and any $s<s_*$,
 \[\frac{1}{n}\vol_X(C\cap  W_*(s))\geq \sum_{\tau\in \Sigma_*}  (s\cdot\Nm_*(\lambda(\tau)))^{1/n}\mult_{Z(\tau)}C\]
 \end{prop}

\begin{proof}  From the proof of the previous proposition, we have
\[\frac{1}{n}\vol_X(C\cap  W_*(s))\geq 2s^{1/n}\cdot \lim_{s\into 0}\int_{C\cap  W_*(s)}i\partial\bar\partial (-N_*^{1/n})\]

For $s$ sufficiently small, $C\cap \tilde W_*(s)$ is a union of pure 1-dimensional analytic sets, each component of which is normalized by a disk $f:\Delta\into C\cap \tilde W_*(s)$.  We may assume $f(0)\in Z(\tau)$ and that $f|_{\Delta^*}$ is an isomorphism onto an open set of $C$.  If $\chi_1,\ldots,\chi_s$ are the dual basis to the generators $\lambda_1,\ldots,\lambda_s$ of $\tau$, then write $m_{j}=\ord f^*q^{\chi_j}$.  Note that $m=\min(m_j)$ is the contribution of the branch $f$ to $\mult_{Z(\tau)}C$.

Now for sufficiently small $s$, we have
\[\int_{C\cap \tilde W_*(s)}i\partial\bar\partial (-N_*^{1/n})=\sum_f\int_{\Delta}f^*i\partial\bar\partial (-N_*^{1/n})\]
but of course
\begin{align*}\int_{\Delta}f^*i\partial\bar\partial (-N_*^{1/n})&\geq \pi\nu(f^*(-N_*^{1/n}),0)\\
&\geq \pi m \nu(-N_*^{1/n},\tau)\\
&=\frac{m}{2}\Nm_*(\lambda(\tau))^{1/n}
\end{align*}
by Lemma \ref{lelong}, and the claim follows.

%

\end{proof}
\begin{remark}\label{exact}In fact we can obtain a more precise estimate if we remember more of the local behavior of $C$.  In the notation of the proof, for the branch $f:\Delta\into C\cap \tilde W_*(s)$ let $t$ be a uniformizer in the local ring of the disk at $0$.    We have
\begin{align*}2\pi\nu(f^*(-N_*^{1/n}),0)&=\lim_{t\into 0}\frac{\prod_i\log |f^*q^{\sigma^i_*}|}{\log^n|t|}\\&=\lim_{t
\into 0}\frac{\prod_i \sum_j\sigma^i_*(m_j\lambda_j)\log|t|}{\log^n|t|}\\&=\Nm_*\left(\sum m_j\lambda_j\right)\end{align*}
which is a version of the multiplicity weighted by the local geometry of $Z(\tau)$.
\end{remark}
\begin{remark} Proposition \ref{hwangtocusp} is also optimal in the sense that the geodesic in Remark \ref{optimal} will realize the equality with the coefficient from Remark \ref{exact}.
\end{remark}
The Lelong number jumps up along smaller $T_*$-orbits in the boundary.  Indeed, for any totally positive $\lambda_i\in\Lambda_*$
\[\Nm_*\left(\sum \lambda_i\right)\geq \sum_i\Nm_*(\lambda_i)\]

Recall that a divisor class is nef modulo the boundary if it intersects positively with any integral curve not contained in the boundary (see for example Section 5 of \cite{cerboplus}).  We thus conclude the
\begin{cor}  For a fixed cusp $*$, let $\rho $ run through the (nonequivalent) one-dimensional cones of $\Sigma_*$ and let $D_\rho$ be the corresponding irreducible boundary divisor.  Then
 \begin{equation}(K_{\bar X}+D)-\frac{n}{2\pi}\sum_{\rho\in \Sigma_*}\left(s_*\Nm_*(\rho)\right)^{1/n}\cdot D_\rho \notag\end{equation}
  is nef modulo the boundary.
\end{cor}
\begin{proof}  Intersecting the divisor $K_{\bar X}+D$ against a curve not lying in the boundary computes a multiple of the volume by \eqref{GB}.
\end{proof}

Provided we have for each cusp $*$ a depth $t_*<s_*$ such that a most $\ell$ of the horoball neighborhoods $ W_*(t_*)$ overlap at any point, then we can apply Proposition \ref{hwangtocusp} to each cusp simultaneously:
 \begin{cor}\label{cor212}For $t_*$ chosen as above,
 \begin{equation}(K_{\bar X}+D)-\frac{n}{2\pi \ell}\sum_*\sum_{\rho\in \Sigma_*}\left(t_*\Nm_*(\rho)\right)^{1/n}\cdot D_\rho \label{divisor}\end{equation}
 is nef modulo the boundary.
 \end{cor}

\begin{cor}\label{nef}For $t_*$ chosen as above,
 \[(K_{\bar X}+D)-\frac{n}{2\pi \ell}\sum_* t_*^{1/n}\cdot D_* \]
 is nef modulo the boundary.
\end{cor}
\begin{proof}By definition, $\Nm_*$ is at least 1 on nonzero integral lattice points.
\end{proof}
In particular if the horoball neighborhoods $ W_*(t_*)$ are all disjoint, then Corollaries \ref{cor212} and \ref{nef} hold with $\ell=1$.
 \section{Hilbert modular varieties}\label{hilbert}
 Fix $F$ a totally real field of degree $n$ with ring of integers $\O_F\subset F$ and denote the real embeddings by $\sigma_i:F\hookrightarrow\R$.  We will typically suppress $\O_F,F$ from the notation as much as possible.  For more a more detailed account of Hilbert modular groups, we refer the reader to \cite{vdg,goren}.
\subsection*{Hilbert modular groups}

To any projective rank 2 module\footnote{Strictly speaking $M$ should be endowed with a symplectic pairing as well.} $M$ over $\O_F$ we can associate a Hilbert modular group $\SL(M)\subset G=\SL_2(\R)^n$ after choosing an isomorphism $M\otimes_{\sigma_i}\R\cong\R^2$ for each embedding $\sigma_i$.  Any such $M$ is isomorphic to $M\cong \O_F\oplus \mathfrak{a}$ for an ideal $\frak{a}\subset \O_F$, and therefore up to conjugation we need only consider $\Gamma(1):=\SL(\O_F\oplus \frak{a})$ where the embedding in $G$ is obtained on the $i$th factor from the embedding $\O_F\oplus \frak{a}\into \R^2$ via $\sigma_i$.

Our eventual goal is to prove Theorem \ref{rmmain} uniformly in the choice of $F$ and $M$ for fixed $[F:\Q]$, and we will use the phrase ``uniformly in $\Gamma(1)$" to mean uniformity in this sense.  Note that Hilbert modular groups exist for non-maximal orders as well, though we do not pursue uniformity in this level of generality here.

For any ideal $\frak{n}\subset \O_F$, we define the congruence subgroups
\[\Gamma_0(\n):=\left\{A\in \Gamma(1)\mid A\equiv \mat{*}{*}{0}{*}\Mod \n\right\}\]
\[\Gamma_1(\n):=\left\{A\in \Gamma(1)\mid A\equiv \mat{1}{*}{0}{*}\Mod \n\right\}\]
of elements $A\in \Gamma(1)$ which fix a primitive line or vector $\Mod \n$, respectively.  Concretely, $\Gamma_0(\n)$ consists of determinant $1$ matrices $\mat{\alpha}{\beta}{\gamma}{\delta}$ for which $\alpha,\delta\in\O_F$, $\beta\in \mathfrak{a}^{-1}$, and $\gamma\in\mathfrak{a}\n$.

The Hilbert modular stack associated to $\Gamma(1)$ is the quotient $X(1):=\Gamma(1)\backslash\H^n$, and we also have level covers
\[X_0(\n):=\Gamma_0(\n)\backslash\H^n\hspace{.5in}X_1(\n):=\Gamma_1(\n)\backslash\H^n\]
$X(1),X_0(\n)$, and $X_1(\n)$ are \emph{a priori} analytic Deligne--Mumford stacks.

Let $\Gamma$ be any of the Hilbert modular groups described above.  The Baily--Borel compactification $(\Gamma\backslash\H^n)^*$ of the coarse space of $\Gamma\backslash\H^n$ is obtained by adding a finite set of points corresponding to the equivalences classes of the rational boundary components of $\H^n$ under the action of $\Gamma$.  The Baily--Borel compactifications are normal projective varieties, so in fact $X(1), X_0(\n), X_1(\n)$ are smooth algebraic Deligne--Mumford stacks.  For our purposes the distinction between $X_1(\n)$ and its coarse space is unnecessary since for $|\Nm\n|$ sufficiently large (uniformly in $\Gamma(1)$), $X_1(\n)$ will have no stabilizers:

\begin{lemma}\label{noell}

For $\n\subset\O_F$ with $|\Nm(\n)|>4^n$, $X_1(\n)$ has no elliptic points.

\end{lemma}

\begin{proof}

Suppose $\gamma\in \Gamma_1(\n)$ fixes a point of $\H^n$. Then the eigenvalues of $\gamma$ must be roots of unity. However, the characteristic polynomial of $\gamma$ is $x^2-\alpha x+1$ where $\alpha\in \O_F$ with $\alpha-2\in \n$. It thus follows that the absolute value of the norm of $(\alpha-2),$ if non-zero, is at least as large as the norm of $\n$. However, since $\alpha$ is the sum of two roots of unity it is of absolute value at most $2$ in any archimedean embedding, and thus has norm at most $4^n$. By our assumption on $\n$ the norm of $(\alpha-2)$ must be 0 and thus $\alpha=2$. It follows that $\gamma$ is unipotent, and for $\gamma$ to have fixed points we conclude that $\gamma$ is the identity element, as desired.

\end{proof}

\subsection*{Moduli of abelian varieties with real multiplication}
We briefly describe the moduli interpretation of the Hilbert modular stacks introduced in the previous subsection.

\begin{definition}  Let $S$ be a scheme.  An \emph{abelian scheme with real multiplication} by $\O_F$ over $S$ is an abelian scheme $A/S$ of (relative) dimension $n$ together with an injection $\O_F\into \End_S(A)$.  We simply say $A/S$ has real multiplication if it has real multiplication by $\O_F$ for some totally real field $F$.
\end{definition}

Again, we restrict ourselves to the case of multiplication by the maximal order here.

Over $\C$, the Hilbert modular stack $X(1)$ associated to $\Gamma(1)=\SL(M)$ is naturally identified with the moduli stack of abelian varieties $A$ with real multiplication by $\O_F$ such that $H_1(A,\Z)\cong M$ as $\O_F$-modules\footnote{Again, strictly speaking $M$ should be endowed with a symplectic pairing, in which case the isomorphism is one of polarized $\O_F$-modules.}.  Analytically, the isomorphism is described as follows; for simplicity take $M=\O_F\oplus \mathfrak{a}$.  To any $\tau=(\tau_1,\ldots,\tau_n)\in\H^n$ we can associate a lattice
\[L: =\O_F
\begin{pmatrix}
1\\
\vdots\\
1
\end{pmatrix}+
\mathfrak{a}\begin{pmatrix}
\tau_1\\
\vdots\\
\tau_n
\end{pmatrix}\subset\C^n\]
where $\O_F$ (and $\mathfrak{a}$) acts on the $i$th factor by multiplication via the embedding $\sigma_i$.  The complex torus $A=\C^n/L$ evidently has an action $\O_F\into\End(A)$ with $H_1(A,\Z)\cong \O_F\oplus\mathfrak{a}$, and it can be shown that any $n$-dimensional complex torus with a (faithful) action by $\O_F$ admits a polarization.

Let $\n\subset\O_F$ be an ideal.  For $A$ an abelian variety with real multiplication by $\O_F$, we define the $\n$-torsion $A[\n]$ to be the sub-group scheme annihilated by $\n$.  For two such abelian varieties $A,A'$ a cyclic $\n$-isogeny is an isogeny $f:A\into A'$ such that the induced map $\End(A)_\Q\into \End(A')_\Q$ is a map of $\O_F$-modules and the sub-group scheme $\ker f$ is cyclic and contained in $A[\n]$.  Over a characteristic 0 algebraically closed field, a $\n$-isogeny is equivalent to specifying a cyclic $\O_F/\n$ submodule of $A[\n]$.  As above, over $\C$ we can identify $X_0(\n)$ (resp. $X_1(\n)$) with the stack of abelian varieties with real multiplication by $\O_F$ and a cyclic $\n$-isogeny from $A$ (resp. a point of $A[\n]$ with annihilator $\n$).

\subsection*{Canonical depths}

The rational boundary components of $\Gamma(1)$ (and its subgroups) are naturally identified with $\P^1(F)$, on which $\SL_2(F)$ acts via the standard action.  For an ideal $\n\subset\O_F$, we would like to show that the canonical depths of the cusps of $\Gamma_0(\n)$ grow uniformly in $|\Nm(\n)|$.  For any fractional ideal $\mathfrak{b}\subset F$, denote by $|\mathfrak{b}|$ the smallest nonzero value of $|\Nm(x)|$ for $x\in \mathfrak{b}$.  Clearly we have $|\mathfrak{b}|\geq |\Nm(\mathfrak{b})|$.

Given two distinct (possibly equivalent) cusps $\xi_1=(\alpha:\beta)$, $\xi_2=(\gamma:\delta)\in\P^1(F)$ of $\Gamma_0(\n)$, we may assume by clearing denominators that both have integral coordinates.  Consider the matrix
\[M=\mat{\alpha}{\gamma}{\beta}{\delta}\]
and set $\Delta=\alpha\delta-\beta\gamma$.  By conjugating
\[M^{-1}\Gamma_0(\n)M\]
we have moved $\xi_1$ to $\infty$ and $\xi_2$ to $0$.  Let $\mathfrak{b}_1=\alpha\mathfrak{a}+(\beta)$ and $\mathfrak{b}_2=\gamma\mathfrak{a}+(\delta)$.

%
%
\begin{lemma}The unipotent stabilizer of $\infty$ in $M^{-1}\Gamma_0(\n)M$ is
\[\mat{1}{\Delta\mathfrak{a}(\mathfrak{b}_1^{-2}\cap\beta^{-2}\n)}{}{1}\]
and that of $0$ is
\[\mat{1}{}{\Delta\mathfrak{a} (\mathfrak{b}_2^{-2}\cap \delta^{-2}\n)}{1}\]

\end{lemma}
\begin{proof}We need to know what matrices $\mat{1}{t}{0}{1}\in G$ conjugate back to elements of $\Gamma_0(\n)$, and we find
\begin{equation}M\mat{1}{t}{0}{1}M^{-1}=\mat{1-\frac{\alpha\beta}{\Delta}t}{\frac{\alpha^2}{\Delta}t}{-\frac{\beta^2}{\Delta}t}{1+\frac{\alpha\beta}{\Delta}t}\end{equation}
which is in $\Gamma(1)$ if and only if $t\in \Delta\mathfrak{a}\mathfrak{b}_1^{-2}$, and additionally in $\Gamma_0(\n)$ if we have $t\in \Delta\mathfrak{a}\beta^{-2}\n$ as well.  Similarly for $0$ and elements $\mat{1}{0}{t}{1}\in G$.
\end{proof}

Now the horoball
\[U(s)=\{z\in\H^n\mid N(z)>1/s\}\]
around $\infty$ is mapped under the inversion $\mat{0}{-1}{1}{0}$ to the horoball
\[U^{-1}(s):=\{z\in\H^n\mid \prod_i|z_i|<s\}\]
and furthermore that $U(s)\cap U^{-1}(s')= \varnothing$ if and only if $ss'\leq 1$.

  The \emph{canonical} horoballs at $\xi_1$ and $\xi_2$ are then given in the above coordinates by
\[U_{\xi_1}(s)=U\left(\frac{s}{|\Delta\mathfrak{a}(\mathfrak{b}_1^{-2}\cap\beta^{-2}\n)|}\right)\]
and
\[U_{\xi_2}(s)=U^{-1}\left(\frac{s}{|\Delta\mathfrak{a}(\mathfrak{b}_2^{-2}\cap\delta^{-2}\n)|}\right)\]

\begin{lemma}\label{isogenydepth}  If $\xi_1$ and $\xi_2$ are distinct (possibly equivalent) cusps of $\Gamma_0(\n)$, the canonical horoballs $U_{\xi_1}(|\Nm(\n)|^{1/2})$ and $U_{\xi_2}(|\Nm(\n)|^{1/2})$ are disjoint.  In particular, every cusp of $\Gamma_0(\n)$ has canonical depth at least $|\Nm(\n)|^{1/2}$.
\end{lemma}
\begin{proof} We know the horoballs $U_{\xi_1}(s_1)$ and $U_{\xi_2}(s_2)$ are disjoint if
\[s_1s_2\leq| \Delta\mathfrak{a}(\mathfrak{b}_1^{-2}\cap\beta^{-2}\n)||\Delta\mathfrak{a}(\mathfrak{b}_2^{-2}\cap\delta^{-2}\n)|\] so it is sufficient to show that the right hand side is at least of size $|\Nm(\n)|$. Now, recall that for two non-zero fractional ideals $I,J$, we have $(I\cap J)^{-1} = I^{-1}+ J^{-1}$, and also that by factorization we have that $(I+J)^2=I^2+J^2$.  Thus
$$(\mathfrak{b}_1^2 + \beta^2\n^{-1})\cdot(\mathfrak{b}_2^2 + \delta^2\n^{-1})  = (\alpha^2\mathfrak{a}^2+\beta^2\n^{-1})\cdot(\gamma^2\mathfrak{a}^2+\delta^2\n^{-1})$$ since we have equality of the two first factors (and the two second factors).  The left hand side contains $\Delta^2\mathfrak{a}^2\n^{-1}$, and

\begin{align*}
| \Delta\mathfrak{a}(\mathfrak{b}_1^{-2}\cap\beta^{-2}\n)||\Delta\mathfrak{a}(\mathfrak{b}_2^{-2}\cap\delta^{-2}\n)|
&=| \Delta\mathfrak{a}(\mathfrak{b}_1^2 + \beta^2\n^{-1})^{-1}||\Delta\mathfrak{a}(\mathfrak{b}_2^2 + \delta^2\n^{-1})^{-1}|\\
&\geq \frac{|\Nm(\Delta\mathfrak{a})|^2}{|\Nm(\mathfrak{b}_1^2 + \beta^2\n^{-1})||\Nm(\mathfrak{b}_2^2 + \delta^2\n^{-1})|}\\
&\geq|\Nm(\n)|
\end{align*}
as desired. 
\end{proof}

Now consider the \'etale cover $X_1(\n)\into X_0(\n)$.  Disjoint precisely invariant horoballs pull back to disjoint precisely invariant horoballs, and by definition of the canonical depth it can only increase in covers.  Thus we have the
\begin{cor} \label{canondepth} Each cusp of $X_1(\n)$ has canonical depth at least $|\Nm(\n)|^{1/2}$.  Moreover, the horoball neighborhoods $W_*(|\Nm(\n)|^{1/2})$ are pairwise disjoint.
\end{cor}

\begin{remark}  Lemma \ref{isogenydepth} is far from optimal.  For example, it is not difficult to show by a similar analysis that for a prime $\p$ the cusps of $X_0(\p)$ that ramify over $X(1) $ have canonical depth $|\Nm(\p)|$, and therefore the same is true for all cusps (since an involution interchanges the ramifying and non-ramifying cusps of $X_0(\p)$, though possibly corresponding to a different choice of $M$).  These horoball neighborhoods will however intersect, albeit at most with multiplicity two.  Furthermore, Corollary \ref{canondepth} is even less optimal because it is pulled back from $X_0(\n)$.  For our purposes we only need the uniform growth.
\end{remark}

Finally, using Corollary \ref{nef} we conclude that a divisor of growing slope on $X_1(\n)$ is nef modulo the boundary, uniformly in $\Gamma(1)$:

\begin{prop} \label{slopegrows}For $\n\subset\O_F$ such that $X_1(\n)$ has no elliptic points,
\begin{enumerate}
\item Each cusp of $X_1(\n)$ has canonical depth at least $|\Nm(\n)|^{1/2}$;
\item \[K_{\bar X_1(\n)}+\left(1-\frac{n}{2\pi}|\Nm(\n)|^{1/2n}\right)D \]
 is nef modulo the boundary.
\end{enumerate}
\end{prop}
\begin{proof}The second statement follows from Corollary \ref{nef}.
\end{proof}

\begin{remark}\label{principal}  By a minor modification of the proof of Lemma \ref{isogenydepth}, we see that Proposition \ref{slopegrows} is true of the principal cover $X(\n)$ with $|\Nm(\n)|^{1/2}$ replaced by $|\Nm(\n)|$.
\end{remark}
Morally Proposition \ref{slopegrows} is already true for $X_0(\n)$ (or even $X(1)$), but these quotients will \emph{always} be orbifolds.  For simplicity we leave the interpretation of Proposition \ref{slopegrows} in this context to the reader, noting that things are substantially simplified by the fact that the elliptic points are isolated.

\subsection*{Injectivity radii} The final ingredient for the proof of the main theorem is some uniform estimates on the injectivity radii of Hilbert modular groups.  Recall that for any discrete $\Gamma\subset G=\SL_2(\R)^n$, the injectivity radius $\rho$ at a point $x\in X=\Gamma\backslash \H^n$ is defined as half the length of the smallest loop through $x$, where the length is taken with respect to the Kobayashi distance on $X$.  We can write
\[\rho_x=\frac{1}{2}\inf_{1\neq \gamma\in\Gamma} d_{\H^n}(\tilde x,\gamma\cdot \tilde x)\]
for $\tilde x\in\H^n$ any lift of $x$.  Equivalently, $\rho_x$ is the largest radius $r$ such that the Kobayashi ball $B(x,r)\subset \H^n$ of radius $r$ around $x$ injects in the quotient $X$.  The global injectivity radius $\rho_X$ is the global infimum:
\[\rho_X:=\inf_{x\in X}\rho_x\]
For a group $\Gamma$ with cusps, $\rho_X=0$ because unipotent elements of $\Gamma$ correspond to homotopy classes of geodesics wrapping the cusps, and the length of these clearly go to 0.  We therefore define the semi-simple injectivity radii by only considering semi-simple elements $\Gamma^{ss}\subset\Gamma$:  for $1\neq \gamma\in \Gamma^{ss}$, define
\[\rho_\gamma:=\frac{1}{2}\inf_{z\in\H^n}d_{\H^n}(z,\gamma\cdot z)\]
and then
\[\rho^{ss}_X:=\inf_{x\in X}\rho^{ss}_x\]

\begin{lemma}  For $\gamma\in G$ a semi-simple element with largest eigenvalue $\lambda$, $$\rho_\gamma\geq \log|\lambda|$$
\end{lemma}
\begin{proof} A semisimple $\gamma\in \SL_2(\R)$ can be conjugated to a scaling by the square $a=|\lambda|^2$ of its largest eigenvalue, in which case
\begin{align*}
\inf_{z}d_\H(z,\gamma z)&=\inf_z d_\H(z,az)\notag\\
&=\inf_z \arcosh\left(1+\frac{(a-1)^2|z|^2}{2a(\Imag z)^2}\right)\notag\\
&\geq \log a
\end{align*}
Since the Kobayashi distance on $\H^n$ is the maximum of the coordinate-wise distances, we are done.
\end{proof}
\begin{cor} \label{bigssinject} $\rho^{ss}_{X_1(\n)}\geq \frac{1}{n}\log|\Nm(\n)|-1$.
\end{cor}
\begin{proof}An element $\gamma\in\Gamma_1(\n)$ reduces to $\mat{1}{*}{0}{1}\Mod \n$; if $\gamma$ is not unipotent, neither eigenvalue is 1.  If $\sigma_i(\alpha)$ is the largest eigenvalue of $\sigma_i(\gamma)$ over all $i$, then since $\alpha-1\in\n$ we have 
\begin{align*}
|\sigma_i(\alpha)|&\geq |\Nm(\alpha-1)|^{1/n}-1\\
&\geq |\Nm(\n)|^{1/n}-1
\end{align*}

\end{proof}
For any cusp $*$ with unipotent stabilizer $\Lambda_*$, we also define at each point $x\in X$
\[\rho^*_{x}:=\frac{1}{2}\inf_{\gamma\in\Lambda_*}d_{\H^n}(\tilde x,\gamma\cdot \tilde x)\]
We'd like to continue thinking of $1/N_*$ as the ``distance" to $*$, but $N_*$ isn't defined on all of $X$.  If $s_*$ is the canonical depth of $*$, we can simply define $N_*$ to be $1/s_*$ outside of $\tilde W_*(s_*)$.

\begin{lemma} \label{biguniinject}$\rho^*_x\geq N_*(x)^{-1/n}\cdot (1+O(N_*(x)^{-1/n}))$.
\end{lemma}
\begin{proof}For a real translation by $a$ in $\H$, we have
\[d_\H(z,z+a)=2\tanh^{-1}\left(\frac{a}{\sqrt{a^2+4(\Imag z)^2}}\right)\]
and so
\[d_\H(z,z+a)=\frac{a}{\Imag z}\cdot(1+O(a/\Imag z))\]
in a region where $\Imag z$ is bounded away from $0$.  Thus for $\lambda\in\Lambda_*$,
\[d_{\H^n}(z,z+\lambda)\geq \frac{\Nm_*(\lambda)^{1/n}}{N_*(z)^{1/n}}\cdot(1+O(\Nm_*(\lambda)^{1/n}/N_*(z)^{1/n}))\]
By definition the smallest nonzero value of $\Nm_*(\lambda)$ is 1, and the result follows.
\end{proof}

\section{Geometric consequences}\label{geometry}  Proposition \ref{slopegrows} alone is enough to prove some interesting geometric results, including Theorem \ref{getsample} and Corollary \ref{green}.  We  choose once and for all smooth toroidal compactifications $\bar X_1(\n)$ for all $X(1)$ and $\n\subset\O_F$ (for which $X_1(\n)$ has no elliptic points).  Each of the following results has a strengthening for \emph{principal} covers $\bar X(\n)$ using Remark \ref{principal} which for the most part we leave implicit.    

First, it follows that the slope of the ``ample modulo $D$" cone of $\bar X_1(\n)$ grows uniformly:
\begin{prop}\label{slope}Assume $X_1(\n)$ has no elliptic points.  For any $\lambda>0$, $K_{\bar X_1(\frak\n)}-(\lambda -1)D$ is ample modulo the boundary provided $|\Nm(\n)|>(\frac{2\pi\lambda}{n})^{2n}$.
\end{prop}
\begin{proof}  $ K_{\bar X_1(\n)}+D$ is ample modulo the boundary, and for $|\Nm(\n)|$ as in the proposition $K_{\bar X_1(\n)}-(t_0-1)D$ is nef modulo the boundary for $t_0>\lambda$ so $K_{\bar X_1(\n)}-(t-1)D$ for any $t\in(0, t_0)$ is ample modulo the boundary.  Here we've used a ``modulo $D$" version of the Kleiman criterion, which can be easily proven along the lines of the usual one, \emph{e.g.} \cite{lazarsfeld}.
\end{proof}
Note that by Lemma \ref{noell} we may uniformly assume $X_1(\n)$ has no elliptic points.  By a general result of Tai \cite{amrt}, an arithmetic quotient $\Gamma\backslash\Omega$ of a bounded symmetric domain has an \'etale cover which is of general type.  Proposition \ref{slope} immediately implies almost all of the torsion covers of Hilbert modular varieties are of general type:
\begin{cor} Assume $X_1(\n)$ has no elliptic points.  Then it is of general type provided $|\Nm(\n)|>\left(\frac{2\pi}{n}\right)^{2n}$.
\end{cor}
\begin{proof}Any divisor that is ample modulo $D$ is big.
\end{proof}
In particular, this will be true of \emph{every} torsion cover without elliptic points for $n\geq 6$ (as in this case $2>(\frac{2\pi}{n})^{2n}$), though by \cite{tsuy} $X(1)$ itself is of general type for $n>6$.

There is also an interesting consequence pertaining to the hyperbolicity of the torsion covers $\bar X_1(\n)$.  Recall that the Green--Griffiths conjecture asserts that for any general type variety $X$, there is a strict subvariety $Z\subset X$ such that every entire map $\C\into X$ has image contained in $Z$.

\begin{cor} \label{greenbody} Every entire map $\C\into \bar X_1(\n)$ has image contained in the boundary provided $|\Nm(\n)|>(2\pi)^{2n}$.
\end{cor}
\begin{proof}  By a theorem of Nadel \cite[Theorem 2.1]{nadel}, an entire map $\C\into \bar X$ into a toroidal compactification of a quotient of a bounded symmetric domain must be contained in the boundary as soon as $K_{\bar X}+(1-1/\gamma )D$ is big, where the sectional curvature is bounded from above by $-\gamma$ (with the normalization $\Ric(h)=-h$).  For us $\gamma=1/n$ is sufficient, and by Proposition \ref{slopegrows} we see that $|\Nm(\n)|>(2\pi)^{2n}$ is enough.  By Lemma \ref{noell}, this rules out elliptic points as well.
\end{proof}
We can deduce from Corollary \ref{greenbody} the genus $0$ and $1$ case of the geometric torsion conjecture.  Our proof in the next section for this special case will be independent of Corollary \ref{greenbody}.
\begin{cor}Every rational or elliptic curve in $\bar X_1(\n)$ is contained in the boundary provided $|\Nm(\n)|>(2\pi)^{2n}$.  The same is true of $\bar X(\n)$ provided $|\Nm(\n)|>(2\pi)^n$. 
\end{cor}
This improves a result of Freitag \cite{freitag} that the statement is true for sufficiently large principal congruence groups with constant depending on $X(1)$.  In the $n=2$ case, a Hilbert modular surface $X$ possesses a canonical smooth model $Y$ that has no $-1$ curves in the cuspidal or elliptic resolutions, but $Y$ is often not minimal.  Hirzebruch and Zagier \cite{hirzclass} conjectured that as long as the surface is irrational, the minimal model is obtained by blowing down ``known" curves---that is, modular curves or curves arising from the desingularization.  Van der Geer proves that in fact the principal congruence covers $Y(\n)$ are minimal \cite{vdgminimal} for $|\Nm(\n)|$ larger than a computable constant depending on $X(1)$, and we obtain a uniform improvement:

\begin{cor}  For $n=2$, $Y_1(\n)$ is minimal provided $|\Nm(\n)|>(2\pi)^4$.    The same is true of $Y(\n)$ provided $|\Nm(\n)|>(2\pi)^2$.
\end{cor}

\section{Geometric torsion:  uniformity in genus}\label{genussection}
Let $C$ be a quasi-projective curve over $k=\C$.  The generic fiber of an abelian scheme $A/C$ yields an abelian variety over the function field $A/k(C)$, and conversely any abelian variety $A/k(C)$ yields an abelian scheme $A/U$ over an open set $U\subset C$.  We define the Mordell--Weil group to be the group of sections, or equivalently the groups of rational sections $A(C)=A(k(C))$, and denote by $A(C)_{tor}$ the torsion subgroup.  An abelian scheme $A/C$ is isotrivial if the fibers $A_c$ are generically isomorphic, or equivalently if $A/k(C)$ is the base-change of an abelian variety over $\C$, and $A/C$ is said to have no isotrivial part if it has no isotrivial isogeny factor.  Note that an abelian scheme $A/C$ with real multiplication has an isotrivial part only if it is isotrivial.  By a folklore result, abelian schemes $A/C$ with no isotrivial part have finitely generated Mordell--Weil group $A(C)$.    

Given our current setup, to prove Theorems \ref{rmmain} and \ref{nogon} it will be enough to show the

\begin{theorem}\label{GTC}  Let $\bar X_1(\n)$ be a smooth toroidal compactification of $X_1(\n)$ for any $n$-dimensional $X(1)$.  For any $g$,  every curve $C\into \bar X_1(\frak{n})$ with (geometric) genus $g(C)<g$ is contained in the boundary for all but finitely many $\n$, uniformly in $X(1)$.
\end{theorem}
In the following we'll use the phrase ``uniformly in $X(1)$" if the statement holds for all $X(1)$ with constant only depending on the dimension $n$ (and not on the choices of the field $F$ and the module $M$).  For any $\lambda>0$, define
\[\ell(K_{\bar X}-\lambda D)=\inf_{C\not\subset D}(K_{\bar X}-\lambda D).C\]
where the infimum is taken over all integral curves not contained in the boundary.  We show that
\begin{prop} \label{getsbig} For any $\lambda,B>0$, we have
\[\ell(K_{\bar X_1(\frak n)}-\lambda D)>B\]
for all but finitely many $\n$, uniformly in $X(1)$.
\end{prop}
Theorem \ref{GTC} then follows from Proposition \ref{getsbig} and Lemma \ref{schwarz} by taking $\lambda=n-1$.
 \begin{proof}[Proof of Proposition \ref{getsbig}]For curves $C\into \bar X_1(\n)$ not contained in the boundary we define
\[d(C):=\sup_{x\in C}\sup_*\frac{1}{N_*(x)}\]
We divide such curves into three categories depending on how close to the boundary they are:
\begin{enumerate}
\item Those meeting the boundary, $d(C)=0$;
\item Those not meeting the boundary with a point far from the boundary, $d(C)\geq E$, for some $E>0$ to be determined;
\item Those entirely lying close to but not touching the boundary, $0<d(C)<E$.
\end{enumerate}

We just need to uniformly show that
\begin{equation}(K_{\bar X_1(\n)}-\lambda D).C> B\mbox{\indent for\indent} |\Nm\n|\gg 0\label{goal}\tag{$*$}\end{equation}
for all $C\not\subset D$.
\vskip 1em
\noindent\underline{Case (1)}:  In this case $D.C>0$, and therefore Proposition \ref{slopegrows} immediately implies \eqref{goal}.

\vskip 1em

\noindent\underline{Case (2)}:  If $d(C)\geq E$, then there is some point $x\in C$ with $N_*(x)\leq E$ for all cusps $*$.  It follows from Corollary \ref{bigssinject} and Lemma \ref{biguniinject} that $\rho_x$ is large in this region uniformly in $|\Nm\n|$ if we take $E$ sufficiently large only depending on $B$.  By \eqref{htgb},
\[K_{\bar X}.C\geq \frac{2}{n}\sinh^2(\rho_x/2)\mult_x C\geq\frac{2}{n}\sinh^2(\rho_x/2)\]
and this is sufficient for \eqref{goal} since $D.C=0$.
\vskip 1em
\noindent\underline{Case (3)}:  For $|\Nm(\n)|$ large, $E<s_*$ for each cusp $*$.  Then any curve $C$ in Case (3) is entirely contained in the horoball neighborhood $\tilde W_*(s)$ for any $E<s<s_*$, which is impossible.  Indeed, by Proposition \ref{relative}
\begin{align*}
\vol(C)=\vol(C\cap W_*(s_*))&\geq\left(\frac{s_*}{s}\right)^{1/n}\cdot\vol_X(C\cap  W_*(s))\\
&\geq \left(\frac{s_*}{s}\right)^{1/n}\cdot\vol(C)
\end{align*}

\end{proof}

The following easy corollary will be useful for proving uniformity in gonality in the next section.  It essentially says that the error term in the Schwarz lemma \eqref{inter} coming from intersections with the boundary becomes negligible sufficiently high in the torsion level tower:

\begin{cor}\label{genusineq}For any $\epsilon>0$ and any curve $C\into \bar X_1(\n)$ not contained in the boundary we have
\[(1-\epsilon)\cdot \frac{1}{2\pi n}\vol (C)\leq 2g(C)-2\]
 for all but finitely many $\n$, uniformly in $X(1)$.
\end{cor}
\begin{proof}Take $\lambda =\frac{n}{\epsilon}-1$.  Then by the proposition and Lemma \ref{schwarz},
\begin{align*} (1-\epsilon)\cdot \frac{1}{2\pi n}\vol(C)&= \frac{1-\epsilon}{n}\left(K_{\bar X}+D\right).C\\
&\leq \frac{1-\epsilon}{n}\left(K_{\bar X}+D\right).C+\frac{\epsilon}{n}(K_{\bar X}-\lambda D).C\\
&= \frac{1}{n}\left(K_{\bar X}-(n-1)D\right).C\\
&\leq2g(C)-2
\end{align*}
for $|\Nm(\n)|\gg 0$.
\end{proof}

\section{Geometric torsion:  uniformity in gonality}\label{gonalitysection}
\subsection*{Preparations}  The idea of the proof of Theorem \ref{rmmaingon} is the same as in Section \ref{genussection}:  we show that the torsion level covers $\bar X_1(\n)$ don't admit maps from $d$-gonal curves $C$  for large $|\Nm\n|$.  Note that given such a map $C\into X_1(\n)^*$, the degree $d$ map $C\into \P^1$ gives us a map into the $d$-fold symmetric product:
\[\P^1\into \Sym^dX_1(\n)^*\]
This can be thought of as Weil restriction of the associated (rational) family of abelian varieties over $C$ down to $\P^1$.
\begin{prop}\label{gonality}  Fix $X(1)$.  Then every rational curve in $\Sym^dX_1(\n)^*$ is contained in a diagonal for all but finitely many $\n$.
\end{prop}

We will again prove the theorem by showing that the only rational curves in $\Sym^d \bar X_1(\n)$ are contained in the boundary or a diagonal.  By the boundary of $\Sym^d \bar X_1(\n)$, we mean the image under the quotient $q:\bar X_1(\n)^d\into\Sym^d\bar X_1(\n)$ of the locus of points that project to the boundary in some projection.

Let $\bar X$ be the toroidal compactification of a quotient $X=\Gamma\backslash\H^n$ of $\H^n$ by a rank 1 lattice $\Gamma$.  The first step is to relate the genus of a curve $C\into \Sym^d\bar X$ to its volume in the same spirit as Lemma \ref{schwarz}.  Of course, as mentioned in Remark \ref{orbifold}, the \emph{orbifold} symmetric product $\mathcal{S}\mathrm{ym}^d X$ is a perfectly valid quotient of $\H^{dn}$, and if we consider a proper map $U\into \mathcal{S}\mathrm{ym}^d X$ from a punctured orbifold curve, then Lemma \ref{schwarz} still applies.  As every curve $C$ in the coarse space $\Sym^d X$ is the coarse space of a minimal orbifold curve mapping to $\mathcal{S}\mathrm{ym}^d X$, we obtain a lower bound to the multiplicity of $C$ along the diagonals and the boundary in terms of the volume.  For clarity we give a second argument only involving the coarse spaces.

Consider a map $\alpha:\P^1\into \Sym^d\bar X$ whose image is not contained in any diagonal, and let $C$ be the normalization of the fiber product
\[\xymatrix{
C\ar[d]_\pi\ar[r]^{\alpha'}&\bar X^d\ar[d]^q\\
\P^1\ar[r]_{\alpha}&\Sym^d \bar X}\]
We may assume $C$ is irreducible.  Define the total ramification
\[\Ram(\pi):=\sum_{p\in C}(r_p-1)\]
where $r_p$ is the order of ramification of $\pi$ at $p$, so that Riemann--Hurwitz says
\[2g(C)-2=d!(-2)+\Ram(\pi)\]
Now $\pi$ can only ramify at points $p$ on a diagonal, meaning $\pi(p)$ projects to the diagonal $\Delta_{\bar X}$ via one of the projections to $\bar X\times \bar X$, and the ramification order is less than $d!$.  Let $C_{ij}=C\into \bar X\times \bar X$ be the map obtained by composing $\alpha'$ with the $ij$th projection.  Thus, we have

\begin{equation}\label{broken}\frac{1}{d!}\left(2g(C)-2\right)\leq \max_{ij}\mult_{\Delta_{\bar X}}C_{ij}\end{equation}

To prove the proposition it is enough to show the following diagonal multiplicity bound:

\begin{prop}\label{diagonalineq}  For any $M>0$ and any curve $C\into \bar X_1(\n)\times \bar X_1(\n)$ not contained in the boundary or the diagonal $\Delta_{\bar X_1(\n)}$ we have 
\[\vol_{X_1(\n)\times X_1(\n)}C\geq M\cdot \mult_{\Delta_{\bar X}} C\]
for all but finitely many $\n$, uniformly in $X(1)$.
\end{prop}
\begin{proof}[Proof of Proposition \ref{gonality} given Proposition \ref{diagonalineq}]  Continuing with the above notation, denote by $C_i=C\into \bar X$ the composition of $\alpha'$ with the projection to the $i$th factor.  Certainly we have
\[\vol_{X\times X}C_{ij}=\vol_XC_i+\vol_X C_j\]
Proposition \ref{diagonalineq} and \eqref{broken} then clearly contradict the asymptotic Schwarz lemma in Corollary \ref{genusineq}.
\end{proof}

\subsection*{Diagonal multiplicity estimate} To prove Proposition \ref{diagonalineq}, we first summarize the setup of \cite{hwangto2} where a diagonal multiplicity estimate for compact curves is proven.  Recall that $d_\H(z,w)$ is the hyperbolic distance function on $\H\times\H$.  For any $R>0$, denote by
\[T_\H(R)=\{(z,w)\in\H\times\H\mid d_\H(z,w)<R\}\]
the radius $R$ tubular neighborhood of the diagonal $\Delta_\H\subset\H\times\H$.  Hwang and To construct a continuous increasing function $\mu:\R\into [-\infty,\infty)$ such that, setting $f=\mu\circ d_\H$,
\begin{enumerate}
\item $\mu$ is supported on $[0,R]$ with $\mu(0)=-\infty$ and $f$ is continuous and finite valued on the complement of $\Delta_\H$;
\item $i\partial\bar\partial [f]\geq -\omega_{\H\times\H}$ on the complement of $\Delta_\H$, and for any point $\xi\in\Delta_{\H\times\H}$,
\[\nu(i\partial\bar\partial [f]+\omega_{\H\times\H},\xi)=L_R\]
\item $\lim_{R\into \infty}L_R=\infty$.
\end{enumerate}
By $[f]$ we mean the distribution associated to $f$, which is sensible since (1) and (2) imply that $f+\phi$ is plurisubharmonic, where $\phi$ is a potential for $\omega_{\H\times\H}$.  In general, for any complex manifold $X$ with a (closed) positive real (1,1) current $\omega$, we say that a continuous locally integrable function $g:X\into[-\infty,\infty)$ is $\omega$-plurisubharmonic ($\omega$-psh) if $i\partial\bar\partial [g]\geq -\omega$.  Thus, $f$ is $\omega_{\H\times\H}$-psh on $\H\times\H$.  

%
%

$f$ allows us to prove a diagonal multiplicity estimate for a \emph{compact} quotient $X=\Gamma\backslash\H$.  Indeed, for $R<\rho_X$, then the quotient of $T_\H(R)$ by the diagonal action of $\Gamma$ embeds as $T_X(R)$ in $X\times X$, and since $f$ is diagonally invariant it descends to a $\omega_{X\times X}$-psh function $f$ on $X\times X$ supported on $T_X(R)$.  For any curve $C\subset X\times X$ we then have (\emph{cf.} \cite{hwangto2})
\begin{equation}
\vol_{X\times X}(C\cap T_X(R))=\int_{C\cap T(R)}i\partial\bar\partial [f] +\omega_{X\times X}\geq L_R\cdot \mult_{\Delta_X} C
\label{bound}\end{equation}

In fact, Hwang and To construct $f$ so that $L_R$ is optimal, but we only need property (4) above.  Since the maximum of plurisubharmonic functions is plurisubharmonic, and the Kobayashi distance $d_{\H^n}$ on $\H^n\times \H^n$ is the maximum of the coordinate-wise distances, properties (1)--(3) continue to hold for $f=\mu\circ d_{\H^n}$ on $\H^n\times\H^n$.

For noncompact quotients $X=\Gamma\backslash\H^n$, the above approach fails because $\rho_X=0$.  The idea is to uniformly introduce a new metric on $\bar X_1(\n)$ so that $\rho_{\bar X_1(\n)}$ is nonzero and growing.  The key point is that for a fixed $F$ the unipotent part of the parabolic stabilizer of a cusp of $X_1(\n)$ is one of only finitely many lattices up to scale, corresponding to ideal classes of $F$.  A new metric in a cuspidal neighborhood can therefore be glued in uniformly in $\n$ (for fixed $F$), and Proposition \ref{relative} will show us that the difference between volumes of curves with respect to the old and the new metric is negligible sufficiently high in the tower.

\subsection*{Modified cuspidal metrics}  
We have seen that each cusp of $\Gamma_1(\n)$ can be conjugated to infinity so that its stabilizer has the form
\[\mat{u}{\lambda}{0}{u^{-1}}\mbox{ for } u\in H, \lambda\in \Lambda\]
where $\Lambda$ is some fractional ideal and $H$ is a subgroup of $\O_F^*$.  By scaling we can assume $N_*=N$.  Denote by $W_{\Lambda,H}(s)$ the quotient of $U(s)=\{z\in\H^n\mid N(z)>1/s\}$ by this group.  If we take $\{\Lambda_i\}$ to be a set of fractional ideals representing the ideal classes of $\Cl(F)$, then it follows that for each cusp $*$ of $X_1(\n)$ the horoball neighborhood $W_*(s)$ is isomorphic to some $W_{\Lambda_i,H}(s)$.  We fix once and for all a smooth toroidal compactification of each $W_{\Lambda_i,\O^*_F}(s)$.  Note that the same fan yields a smooth toroidal compactification of $W_{\Lambda_i,H}(s)$ for any finite index $H\subset \O_F^*$ and the resulting map $q:\tilde W_{\Lambda_i,H}(s)\into \tilde W_{\Lambda_i,\O_F^*}(s)$ is \'etale.  Using these fixed cuspidal resolutions, it will therefore be the case that for each cusp $*$ of $X_1(\n)$ we have $\tilde W_*(s)\cong \tilde W_{\Lambda_i,H}(s)$ for some $i$ and $H$.

For given $\Lambda$ and $s>0$, we can endow $\tilde W(s):=\tilde W_{\Lambda,\O_F^*}(s)$ with a Riemannian metric which agrees with the usual one outside of a relatively compact set.  Given $R>0$ we can find $\epsilon>0$ and scale the distance function by a function supported on $\tilde W(s-\epsilon)$ such that the length of any loop in $\tilde W(s-\epsilon)$ has length at least $R$.  Thus, the new distance function $d:\tilde W(s-\epsilon)\times \tilde W(s-\epsilon)\into \R$ is smooth in a radius $R$ neighborhood of the diagonal.  If we fix a K\"ahler form $\omega_0$ on $\tilde W(s)$ of bounded total volume, then $\mu\circ d$ is $B\omega'_0$-psh on $ \tilde W(s-\epsilon)\times \tilde W(s-\epsilon)$ for some constant $B>0$, where $\omega_0'$ is the sum of the pullbacks of $\omega_0$ along each projection.

We know that in polycylinder coordinates at the boundary the usual K\"ahler form $\omega_{W_\Lambda(s)}$ looks like
\[\omega_{W(s)}\sim \sum_{i}\frac{|dq_i|^2}{|q_i|^2\log^2|q_i|}+\sum_j|dq_j|^2\]
and it then follows that
\[\omega_0=o(\omega_{W(s)})\mbox{ as }N(z)\into \infty\]
so in fact $\mu\circ d$ is also $B \omega_{W(s-\epsilon)\times W(s-\epsilon)}$-psh for some $B>0$.  

Now pull the new metric back to $\tilde W_{H}(s):=\tilde W_{\Lambda,H}(s)$.  By Lemma \ref{biguniinject}, we can assume that $W_H(s)$ already has large injectivity radius (with respect to the usual metric) at points for which $N(z)$ is sufficiently small.  By the above construction we can ensure that this new metric has distance function $d:\tilde W_{H}(s)\times \tilde W_{H}(s)\into \R$ which is smooth in a radius $R$ tube around the diagonal, and that $\mu\circ d$ is $B \omega_{W_H(s)\times W_H(s)}$-psh for the \emph{same} $B>0$.  Note that as $q$ is \'etale, $d$ acquires no new zeroes at the boundary.

Applying the above procedure to each $\Lambda=\Lambda_i$ and gluing into $X_1(\n)$, we have therefore proven the

\begin{lemma}  Fix $X(1)$.  For each $R>0$, there are constants $s_R,B_R>0$ such that for $|\Nm(\n)|\gg 0$, there is a function $f:\bar X_1(\n)\times \bar X_1(\n)\into[-\infty,\infty)$ satisfying:
\begin{enumerate}

\item $f$ is supported on
\[T_{X_1(\n)}(R)\cup \bigcup_*\tilde W_*(s_R)\times\tilde W_*(s_R) \]
and continuous on the complement of the diagonal;
\item $f$ is $B_R\omega_{X_1(\n)\times X_1(\n)}$-psh and for any point $\xi\in\Delta_{\bar X_1(\n)}$,

\begin{equation}\label{ineqlelong} \nu(i\partial\bar\partial f+B_R\omega_{X_1(\n)\times X_1(\n)},\xi)\geq L_R\end{equation}

\item $f$ is $\omega_{X_1(\n)\times X_1(\n)}$-psh outside of $\bigcup_*\tilde W_*(s_R)\times\tilde W_*(s_R)$;
\item $\lim_{R\into \infty}L_R=\infty$.
\end{enumerate}
\end{lemma}

  Of course we also have 
\[ \nu(i\partial\bar\partial f+\omega_{X_1(\n)\times X_1(\n)},\xi)=L_R\]
 for any point $\xi \in\Delta_{\bar X_1(\n)}$ outside of $\bigcup_*\tilde W_*(s_R)\times\tilde W_*(s_R)$.
  \subsection*{Conclusion of the proof}
  We are now ready to prove the diagonal multiplicity inequality:
  
  \begin{proof}[Proof of Proposition \ref{diagonalineq}]
  Given $M>0$, choose $R$ such that $L(R)>2M$.  For simplicity write $ X= X_1(\n)$ and let 
  \[\tilde W_R=\bigcup_*\tilde W_*(s_R)\times\tilde W_*(s_R)\subset \bar X\times\bar X\]
  For any curve $C\into \bar X\times \bar X$ not contained in the diagonal or the boundary, we have for $|\Nm(\n)|\gg 0$
  \begin{align*}
 B_R\cdot \vol_{X\times X}C&=\int_C B_R\cdot \omega_{X\times X}\\
 &=\int_C i\partial\bar\partial f+ B_R\cdot \omega_{X\times X}\\
  &\geq \int_{C\cap X\smallsetminus \tilde W_R}(B_R-1)\omega_{X\times X}+L_R\cdot \mult_{\Delta_{\bar X}}C\\
  &=(B_R-1)\cdot \vol_{X\times X}(C\smallsetminus \tilde W_R)+L_R\cdot \mult_{\Delta_{\bar X}}C
  \end{align*}

  and therefore
  \begin{equation}\label{almostthere}\vol_{X\times X}(C\smallsetminus \tilde W_R)+B_R\cdot\vol_{X\times X}(C\cap \tilde W_R)\geq 2M\cdot \mult_{\Delta_{\bar X}}\end{equation}
  Once again letting $C_i=C\into \bar X$ be the projection to the $i$th factor, we have by Proposition \ref{relative}
  \begin{align*}\vol_{X\times X}(C\cap \tilde W_R)&\leq \sum_{i=1,2} \sum_*\vol_{X}(C_i\cap W_*(s_R)) \\
&\leq \sum_{1,2}\sum_* \left(\frac{s_R}{s_*}\right)^{1/n}\vol_{X}(C_i\cap W_*(s_*)) \\
&\leq \frac{1}{B_R}\cdot \vol_{X\times X} C
\end{align*}
where in the last line we've taken $|\Nm(\n)|$ large enough so that $s_*>s_R\cdot B_R^n$, using Proposition \ref{slopegrows}.  Combining this with \eqref{almostthere}, we obtain
\[\vol_{X\times X}C\geq M\cdot \mult_{\Delta_{\bar X}C}\]
for $|\Nm(\n)|\gg 0$, as desired.
  \end{proof}

\bibliography{biblio.hmv}
\bibliographystyle{alpha}
 \end{document}